\documentclass{article}

\usepackage[12pt]{extsizes}
\usepackage{cmap}
\usepackage[utf8]{inputenc}
\usepackage[T2A]{fontenc}
\usepackage[english]{babel}
\usepackage{amsmath}
\usepackage{amsthm}
\usepackage{amssymb}
\usepackage{amsfonts}
\usepackage{euscript}
\usepackage{enumerate}
\usepackage{hyperref}
\usepackage[affil-it]{authblk}

\usepackage[all]{xy}
\usepackage{epigraph}

\newcommand{\ds}{\displaystyle}
\newcommand{\RN}{\mathbb{R}} 
\newcommand{\CN}{\mathbb{C}} 
\newcommand{\ZN}{\mathbb{Z}} 

\newcommand{\eps}{\ensuremath\varepsilon}

\newcommand{\ovl}{\overline}

\newcommand{\gr}{\mathop{\mathrm{gr}}\nolimits}

\newcommand{\SL}{\mathop{\mathrm{SL}}\nolimits}

\renewcommand{\Im}{\mathop{\mathrm{Im}}\nolimits}

\newcommand{\SU}{\mathop{\mathrm{SU}}\nolimits}

\newcommand{\GL}{\mathop{\mathrm{GL}}\nolimits}
\newcommand{\ad}{\ensuremath\operatorname{ad}}
\newcommand{\id}{\ensuremath\operatorname{id}}

\newcommand{\Span}{\operatorname{Span}}

\newcommand{\mf}{\mathfrak}

\newcommand{\mc}[1]{\mathcal{#1}}

\renewcommand{\Re}{\operatorname{Re}}
\newcommand{\Ind}{\operatorname{Ind}}
\newcommand{\sign}{\operatorname{sign}}
\newcommand{\End}{\operatorname{End}}
\newcommand{\Aut}{\operatorname{Aut}}
\newcommand{\HC}{Harish-Chandra }

\author{Daniil Klyuev}
\title{On unitarizable \HC bimodules for deformations of Kleinian singularities.}
\begin{document}
\newtheorem{thr}{Theorem}[section]
\newtheorem*{thr*}{Theorem}
\newtheorem{lem}[thr]{Lemma}
\newtheorem*{lem*}{Lemma}
\newtheorem{cor}[thr]{Corollary}
\newtheorem*{cor*}{Corollary}
\newtheorem{prop}[thr]{Proposition}
\newtheorem*{prop*}{Proposition}
\newtheorem{stat}[thr]{Statement}
\newtheorem*{stat*}{Statement}
\theoremstyle{definition}
\newtheorem{defn}[thr]{Definition}
\theoremstyle{remark}
\newtheorem{rem}[thr]{Remark}
\newtheorem*{rem*}{Remark}
\newtheorem{example}[thr]{Example}
\maketitle
\begin{abstract}
The notion of a \HC bimodule, i.e. finitely generated $U(\mf{g})$-bimodule with locally finite adjoint action, was generalized to any filtered algebra in a work of Losev [Ivan Losev, Dimensions of irreducible modules over W-algebras and
Goldie ranks. arXiv:1209.1083]. Similarly to the classical case we can define the notion of a unitarizable bimodule. We investigate a question when the regular bimodule, i.e. the algebra itself, for a deformation of Kleinian singularity of type $A$ is unitarizable. We obtain a partial classification of unitarizable regular bimodules.


\end{abstract}
\tableofcontents
\section{Introduction}
We say that a $\ZN_{\geq 0}$-filtered associative algebra $\mc{A}$ over $\CN$ with unit is almost commutative if $\gr\mc{A}$ is commutative. Let $\mc{A}$ be an almost commutative filtered algebra, $\tau$ be a linear antiinvolution of $\mc{A}$ that preserves the filtration. Fix $d>0$ such that for all $i,j\geq 0$ we have $[\mc{A}_{\leq i},\mc{A}_{\leq j}]\subset\mc{A}_{\leq i+j-d}$.

\begin{defn}[~\cite{Lo2014},~6.1.2]
Suppose that $M$ is a $\ZN_{\geq 0}$-filtered $\mc{A}$-module. We say that $M$ is a Harish-Chandra $(\mc{A},\tau)$-module if
\begin{enumerate}
\item
$\gr M$ is finitely generated over $\gr\mc{A}$.
\item
For every $a\in\mc{A}_{\leq i}$ with $\tau(a)=-a$ we have $aM_{\leq j}\subset M_{\leq i+j-d}$.
\end{enumerate}
\end{defn}
\begin{example}
Suppose that $\mc{B}$ is an almost commutative algebra, $\mc{A}=\mc{B}\otimes\mc{B}^{opp}$, $\tau(b_1\otimes b_2)=b_2\otimes b_1$. In this case a Harish-Chandra $(\mc{A},\tau)$-module is called a Harish-Chandra $\mc{B}$-bimodule.
\end{example}
\begin{example}
\label{ExClassicalAndATau}
We can link this definition to classical Harish-Chandra modules. Suppose that $\mf{g}$ is a reductive Lie algebra, $\mc{A}=U(\mf{g})$ with natural filtration. Suppose that $\tau$ is an antiinvolution of $\mf{g}$, then $\mf{k}=\mf{g}^{-\tau}$ is a reductive subalgebra of $\mf{g}$. From  $\tau\colon\mf{g}\to\mf{g}$ we get $\tau\colon \mc{A}\to\mc{A}$. It is proved in~\cite{Lo2014} that a Harish-Chandra $(\mc{A},\tau)$-module is the same as a $(\mf{g},\mf{k})$-module, which means finitely generated $\mf{g}$-module with locally finite action of $\mf{k}$.
\end{example}
We will work in the following setting:
\begin{defn}
Suppose that $A$ is a graded algebra. We say that $(\mc{A},\chi)$ is a filtered deformation of $A$ if $\mc{A}$ is a filtered algebra and $\chi$ is an isomorphism $\gr\mc{A} \to A$.
\end{defn}
Suppose that $A$ is a graded Poisson algebra such that the Poisson bracket has degree $-d$, $(\mc{A},\chi)$ is a filtered deformation of $A$. Suppose that $[\mc{A}_{\leq i},\mc{A}_{\leq j}]\subset\mc{A}_{\leq i+j-d}$. From the commutator on $\mc{A}$ we get a Poisson bracket on $\gr\mc{A}$.
\begin{defn}
If $\chi$ sends the Poisson bracket on $\gr\mc{A}$ to the Poisson bracket on $A$ we say that $(\mc{A},\chi)$ is a quantization of $A$.
\end{defn}

We are interested in quantizations of Kleinian singularities. Suppose that $\Gamma$ is a finite subgroup of $\SL(2,\CN)$, $\mc{A}$ is a quantization of $\CN[u,v]^{\Gamma}$. In section~\ref{SecHCModules} we will classify conjugacy classes of antiinvolutions of $\mc{A}$. We will be working with Harish-Chandra $(\mc{A},\tau)$-modules and $\mc{A}$-bimodules.

Our motivation for working with Harish-Chandra bimodules over deformations of Kleinian singulariries is the connection between classical \HC bimodules over $U(\mf{g})$ and \HC bimodules over deformations of Kleinian singularities given by restriction functors. Namely, let $e,f,h$ be an $\mf{sl}_2$-triple in $\mf{g}$, $S=e+\ker\ad f$ be the Slodowy slice. We can attach to $e$ an algebra $\mc{W}$, a certain filtered deformation of $\CN[S]$~(\cite{Pr}). In~\cite{Lo2011} Losev constructed a restriction functor from the category of \HC bimodules over $U(\mf{g})$ to the category of $Q$-equivariant \HC bimodules over $\mc{W}$, where $Q$ is a centralizer of $\{e,f,h\}$. Suppose that the Dynkin diagram $\Phi$ corresponding to the group $G$ is simply laced. Let $\Gamma$ be a finite subgroup of $\SL(2,\CN)$ that corresponds to $\Phi$. If we take $e$ in the subregular orbit then the Slodowy slice $S$ with the map $S\to \mf{g}/G\cong \mf{h}/W$ is a universal deformation of $\CN[x,y]^{\Gamma}$.



Suppose that $\tau$ is an antiinvolution of a reductive Lie algebra $\mf{g}$, $\mf{k}=\mf{g}^{-\tau}$ is the corresponding reductive subalgebra. Consider an $\mf{sl}_2$-triple $e,f,h$ in $\mf{g}$ such that $\tau e=e$, $\tau f=f$, $\tau h=-h$. A more general restriction functor constructed in~\cite{Lo2014},~6.1.2 sends $(\mf{g},\mf{k})$-modules to $(\mc{W},\tau)$ \HC modules, where $\tau$ is antiinvolution of $\mc{W}$ induced from $\tau$.

Now we move to our main object of study. Suppose that $\tau$ commutes with the standard conjugation on $\mf{g}$. Then the composition of $-\tau$ and the conjugation is an antilinear involution of $\mf{g}$. The space of fixed points of this antilinear involution is a real form $\mf{g}_{\RN,\tau}$ of $\mf{g}$.

\begin{defn}
Suppose that $V$ is a $(\mf{g},\mf{k})$-module. We say that $V$ is unitarizable if there exists a positive definite Hermitian form on $V$ such that $\mf{g}_{\RN,\tau}$ acts by anti-Hermitian operators.
\end{defn}

This definition generalizes to \HC $(\mc{A},\tau)$-modules.

\begin{defn}
Suppose that $\mc{A}$ is an almost commutative algebra, $\tau$ is an antiinvolution on $\mc{A}$, $r$ is an antilinear involution on $\mc{A}$ such that $r\tau=\tau r$, $V$ is a \HC $(\mc{A},\tau)$-module. We say that $V$ is unitarizable if there exists a positive definite Hermitian inner product $(\cdot,\cdot)$ on $V$ such that $(au,v)=(u,r\tau(a)v)$ for all $a\in\mc{A}$, $u,v\in V$.
\end{defn}

The restriction functor is expected to send unitarizable modules to unitarizable modules.   This gives motivation for studying unitarizable \HC $(\mc{A},\tau)$-modules and $\mc{A}$-bimodules.

Another motivation comes from physics: the question of unitarizablity of regular bimodule for deformations of Kleinian singularities of type $A$ appears in paper~\cite{BPR} that discusses connections between deformation quantizations and three-dimensional superconformal field theories. 

\begin{rem}
In our situation there are new feature compared to classical \HC modules: Hermitian and unitary structure on a given irreducible \HC module is not necessarily unique, see Remark~\ref{RemManyForms} But they may depend on finitely many parameters. Also there exist unitarizable non-semisimple bimodules, see Remark~\ref{RemNonSemisimplicity}.
\end{rem}

\begin{rem}
Out situation is similar to~\cite{Seth} that studies invariant Hermitian forms on representations of rational Cherednik algebras.
\end{rem}

The article is organized as follows. We are working in the case $\Gamma=C_n$: $\mc{A}$ is a deformation of $\CN[x,y]^{C_n}=\CN[x^n,y^n,xy]$. In section~\ref{SecHCModules} we classify aniinvolutions and antilinear involutions on $\mc{A}$ and study unitarizable $(\mc{A},\tau)$-modules for a certain $\tau$.

In section~\ref{SecHCBimodules} we study irreducible unitarizable \HC bimodules over $\mc{A}$. We first recall the classification of irreducible unitarizable bimodules in case $n=2$. For $n>2$ we restrict our attention to the regular bimodule.

Our main result is a partial classification of unitarizable regular bimodules in the case $\Gamma=C_n$. Recall~\cite{Ho} that quantizations of $\CN[x,y]^{C_n}$ are in one-to-one correspondence with polynomials $P(x)$ of degree $n$ with fixed leading coefficient: to $P(x)$ corresponds the algebra generated by $e,f,h$ with relations $[h,e]=2e$, $[h,f]=-2f$, $ef=P(h-1)$, $fe=P(h+1)$. 

When $P(x)\in\RN[ix]$, there exists an antilinear involution $r$ on $\mc{A}$ such that $r(e)=-f$, $r(f)=-e$, $r(h)=-h$. When we replace $P$ with $\lambda P$ we get the same antilinear involution for $\lambda>0$ but a different anilinear involution for $\lambda<0$, so our answer will depend on the sign of leading coefficient of $i^{n}P$.

The result can be formulated as follows:
\begin{thr}
\label{ThrMain}
\begin{enumerate}
\item
Suppose that $P(x)$ has at least three roots $\alpha$ with multiplicities with $|\Re\alpha|<1$. Then the regular bimodule is unitarizable.
\item
Suppose that $n=2m$, $P(x)$ has leading coefficient $(-1)^m$. Then
\begin{enumerate}
\item
If $P(x)$ has a root $\alpha$ with $|\Re\alpha|<1$ then the regular bimodule is unitarizable.
\item
If for all roots $\alpha$ of $P(x)$ one has $|\Re\alpha|>1$ then the regular bimodule is not unitarizable.
\end{enumerate}
\end{enumerate}
\end{thr}
\begin{rem*}
This gives the complete answer for $P(x)=(-x^2)^m+\ldots$ that has no roots with real part $1$.
\end{rem*}
\begin{rem*}
In the case when $n$ is odd the number of roots $\alpha$ with $|\Re\alpha|<1$ is odd and $P(x)$ has a purely imaginary root. Hence if $P(x)$ has another root $\alpha$ with $|\Re\alpha|<1$ then the regular bimodule is unitarizable.
\end{rem*}

The proof of this theorem uses analytic lemmas that are stated and proved in appendix. After proving this theorem we give other proofs of unitarizability in certain cases.
\subsubsection{Acknowledgments.}
I am grateful to Ivan Losev for formulation of the problem, Ivan Losev and Pavel Etingof for stimulating discussions and remarks on the previous versions of this paper, Pavel Etingof and Fedor Petrov for providing additional proofs of positivity of certain traces.
\section{Harish-Chandra modules}
\label{SecHCModules}
\subsection{Classification of antiinvolutions}

Suppose that $\mc{A}$ is a quantization of $A=\CN[x,y]^{C_n}$. Let $s$ be an involution of $\mc{A}$. Then $\gr s$ is an involution of $A$ that preserves the Poisson bracket. If $s$ is an antiinvolution then $\gr s$ changes sign of the Poisson bracket. If $s$ is an antilinear involution then $\gr s$ is an antilinear involution that preserves the Poisson bracket. In this case the composition of $s$ and the standard conjugation is an automorphism of $A$.

\begin{lem}
Every homogeneous automorphism of $A$ is given by a homogeneous automorphism of $\CN[x,y]$, i.e. $\Aut A=N_{\GL_2(\CN)}(C_n)/C_n$.
\end{lem}
\begin{proof}
Let $\phi$ be a homogeneous automorphism of $A=\CN[e,f,h]/(ef-h^n)$. Suppose that $n>2$. In this case $A_2=\CN h$. It follows that $\phi(h)=ah$ for some $a\in \CN$, $a\neq 0$. Therefore $\phi(e)\phi(f)=a^n h^n$. Since $\phi(e),\phi(f)\in A_n$ it is easy to see that $\phi(e)=be$ or $\phi(e)=bf$ for some $b\in \CN$. From this we deduce that $\phi$ is given by automorphism of $\CN[x,y]$.

Suppose that $n=2$. In this case $\phi$ is defined by $\phi|_{A_2}$. From $ef=h^2$ we deduce that $\phi\in \operatorname{O}_3(\CN)=\GL_2(\CN)/C_2$, hence $\phi$ is given by automorphism of $\CN[x,y]$.
\end{proof}

For $n>2$ we have $N(C_n)=\{\begin{pmatrix}
a & 0\\0 & d
\end{pmatrix}\mid a,d\in \CN\}\cup \{\begin{pmatrix}
0 & b\\c & 0
\end{pmatrix}\mid b,c\in \CN\}$.

Involutions of $A$ are given by elements of order $2$ in $N(C_n)/C_n$. It is easy to see that they are given by $(ad)^{2n}=a^2d^2=1$ or $(bc)^n=(bc)^2=1$. Denote $e^{\frac{\pi i}n}$ by $\eps$. We get the following elements:
\begin{enumerate}
\item
$\id$. The identity automorphism lifts to the identity automorphism, so we do not consider it below.
\item
$\begin{pmatrix}
1 & 0\\
0 & -1
\end{pmatrix}$. We get the  involution $e\mapsto e$, $f\mapsto (-1)^n f$, $h\mapsto -h$.
\item
$\begin{pmatrix}
\eps & 0 \\
0 & s\eps^{-1}
\end{pmatrix}$, where $s=\pm 1$. We get the involution $e\mapsto -e$, $f\mapsto -s^nf$, $h\mapsto sh$.
\item
$\begin{pmatrix}
0 & b\\
\pm b^{-1}& 0
\end{pmatrix}$, $b\in \CN$. In case $n$ is odd we have $\pm b^{-1}=b^{-1}$. Since every quantization has an automorphism corresponding to $\begin{pmatrix}
a & 0\\
0 & a^{-1}
\end{pmatrix}$ we can consider these elements up to conjugation. So we have two matrices $\begin{pmatrix}
0 & 1\\
s & 0
\end{pmatrix}$, $s=\pm 1$ and $s=1$ in case $n$ is odd. We get the involution $e\mapsto f$, $f\mapsto s^ne=e$, $h\mapsto sh$.
\end{enumerate}

Suppose that $\tau_0$ is one of these involutions. Suppose that $\mc{A}$ has an involution or antiinvolution $\tau$ such that $\gr\tau=\tau_0$. It is not hard to see that in this case we can choose generators $e,f,h$ of $\mc{A}$ such that $\tau$ acts on $\Span(e,f,h)$ as $\tau_0$. Hence we can say when $\tau_0$ lifts to $\tau$:
\begin{enumerate}
\item
The involution $e\mapsto e$, $f\mapsto (-1)^n f$, $h\mapsto -h$ lifts to an antiinvolution when $P(x-1)=(-1)^n P(1-x)$, in other words $P(x)=(-1)^nP(-x)$.
\item
The involution $e\mapsto -e$, $f\mapsto -s^nf$, $h\mapsto sh$ lifts to an involution in case $s=1$ for all $\mc{A}$ and lifts to an antiinvolution in case $s=-1$, when $(-1)^n P(x)=P(-x)$.
\item
The involution $e\mapsto f$, $f\mapsto e$, $h\mapsto sh$ lifts to an antiinvolution in case $s=1$  for all quantizations and lifts to an involution in case $s=-1$, $n$ even, when $P(x)=P(-x)$.
\end{enumerate}
Whenever $P(x)=(-1)^nP(-x)$ denote by $\tau$ an antiinvolution $e\mapsto e$, $f\mapsto (-1)^n f$, $h\mapsto -h$.

Arguing similarly we see that real forms are classified by elements of $N_{\SL(2)}(\Gamma)/\Gamma$ up to $A\sim B^{-1}A\ovl{B}$, $B\in N_{\SL(2)}(\Gamma)$. If \[B\in H=\{\begin{pmatrix}
a & 0\\0 & d
\end{pmatrix}\mid ad=1\}\] then this equivalence lifts to equivalence in quantization. Thus the lift when it exists is unique.

Recall that $A= \begin{pmatrix}
a & 0\\0 & d
\end{pmatrix}$ or $A=\begin{pmatrix}
0 & b\\c & 0
\end{pmatrix}$. 

In the first case we have $a\ovl{a}=d\ovl{d}=ad=1$. It is easy to see that all such matrices are equivalent by elements of $H$. In the second case we have $(b\ovl{c})^n=-bc=1$. It is easy to see that all such matrices are equivalent by elements of $H$.

So we have two antilinear involutions of $A$ up to a conjugation. The first is the standard complex conjugation. The second is $e\mapsto -f$, $f\mapsto -e$, $h\mapsto -h$ when $n$ is even, $e\mapsto i^n f$, $f\mapsto i^n e$, $h\mapsto -h$ when is $n$ odd. Denote this involution by $r$. In the case when $n$ is even $r$ lifts to an antilinear involution when $\ovl{P}(-x)=P(x)$.

In case when $n$ is even and $P(x)=P(-x)\in \RN[x]$ we have $r\tau=\tau r$. In this case $r\tau$ sends $h$ to $h$, $e$ to $-f$, $f$ to $-e$.  In case when $n$ is odd we have $r\tau\neq\tau r$. So we will classify unitarizable irreducible modules in case when $n$ is even.

\subsection{Classification of unitarizable irreducible \HC modules}
\label{SubSecIrredModules}
\begin{thr}
In the case when $n$ is even irreducible \HC $(\mc{A},\tau)$-modules are in one-to-one correspondence with arithmetic progressions with difference $2$ that start at a root of $P$ plus $1$ or $-\infty$ and end at root of $P$ minus $1$ or $\infty$. In the case when $n$ is odd irreducible \HC $(\mc{A},\tau)$-modules are in one-to-one correspondence with arithmetic progressions with difference $n$ that start at a root of $P$ plus $1$ and end at root of $P$ minus $1$ or $\infty$. Namely, the arithmetic progression is the set of weights of $V$. 
\end{thr}
\begin{proof}
From the definitions and the equality $\tau h=-h$ we see that $h$ acts locally finitely on any \HC $(\mc{A},\tau)$-module. Suppose that $V$ is an irreducible \HC $(\mc{A},\tau)$-module. Let $v\in V$ be an eigenvector of $h$: $hv=\lambda_0 v$. It is easy to see that $\sum_{k\geq 0}(\CN e^k v+\CN f^k v)$ is a submodule of $V$. 

Hence $V=\oplus_{a\leq k\leq b}V_{\lambda_0+kn}$, where $a,b\in ZN\cup{\pm\infty}$, $V_{\lambda}$ is a one-dimensional eigenspace $h$. Suppose that $a\neq-\infty$, $v\in V_{\lambda_0+an}$. Then $fv=0$, hence $P(h-1)v=efv=0$. We deduce that $P(\lambda_0+an-1)=0$. If $b\neq \infty$ we similarly have $P(\lambda_0+bn+1)=0$. So the set of weights of $V$ is an arithmetic progression that begins at a root of $P$ plus $1$ or $-\infty$ and ends at a root of $P$ minus $1$ or $\infty$. On the other hand from such arithmetic progression we get $V$ in a straightforward way and define $V_{\leq k}=\oplus_{|\lambda-\lambda_0|\leq k}V_{\lambda}$. Thus $V$ becomes a filtered module that satisfies the definition of s \HC module in the case when $n$ is even. 

When $n$ is odd we have $\tau(f)=-f$. Suppose that $V$ is a \HC $(\mc{A},\tau)$-module. Then $f,h\in \gr\mc{A}$ act on $\gr V$ as zero. It follows that $\gr V$ is a finitely generated $\CN[e]$-module. Therefore the set of weights of $V$ is bounded below. On the other hand if the set of weights of $V$ is bounded below we define $V_{\leq k}=\oplus_{\lambda\leq k}V_{\lambda}$. With this filtration $V$ satisfies the definition of a \HC module. We deduce the theorem.
\end{proof}

Now we turn to the question when the irreducible module $V$ with the set of weights $\Lambda$ is unitarizable. We assume that $n$ is even. When both $r$ and $\tau$ are defined we get $P(x)=Q(-x)=\ovl{P}(x)$, hence $P(x)\in\RN[x]$. 
\begin{thr}
Suppose that $V$ is an irreducible \HC $(\mc{A},\tau)$-module. Then $V$ is unitarizable if and only if
\begin{enumerate}
\item
$\Lambda\subset\RN$. 

Let $\lambda_m$ be the smallest element of $\Lambda$ if it exists. Denote by $\Lambda'$ the set $\Lambda\setminus\{\lambda_m\}$, otherwise (if $\Lambda$ is not bounded below) $\Lambda'=\Lambda$.
\item
$P(\lambda-1)<0$ for all $\lambda\in\Lambda'$.
\end{enumerate} 
\end{thr} 
\begin{proof}
A form $(\cdot,\cdot)$ is invariant if and only if $(u,hv)=(hu,v)$ and $(u,ev)=-(fu,v)$ for all $u,v\in V$. The first condition tells us that $\Lambda\subset \RN$. The second condition can be checked for $u\in V_{\lambda}$, $v\in V_{\lambda-2}$, $\lambda\in \Lambda'$. Since $V_{\lambda}$ is one-dimensional we can assume that $v=fu$. We get $(u,efu)=-(fu,fu)$. So $(u,u)=\frac{-1}{P(\lambda-1)}(fu,fu)$ for all $u\in V_{\lambda}$, $\lambda\in \Lambda'$. This defines $(\cdot,\cdot)$ uniquely. This form is positive definite if and only if $P(\lambda-1)<0$ for all $\lambda\in \Lambda'$. We deduce the theorem.
\end{proof}
\paragraph{Example.}
Suppose that $n=2$. There are several cases.
\begin{enumerate}
\item
Suppose that $P$ has two conjugate complex roots. If the leading coefficient of $P$ is negative then unitarizable irreducible \HC modules are in one-to-one correspondence with elements of $\RN/2\ZN$. In the other case there are no unitarizable irreducible \HC modules.
\item
Suppose that $P$ has two real roots $\alpha\leq \beta$.
\begin{enumerate}
\item
Suppose that the leading coefficient of $P$ is positive. If $\beta-\alpha$ is an even positive integer then there is one unitarizable irreducible $(\mc{A},\tau)$-module, a finite-dimensional module. In the other case there are no unitarizable irreducible $(\mc{A},\tau)$-modules.
\item
Suppose that the leading coefficient of $P$ is negative. In this case we always have two unitarizable irreducible $(\mc{A},\tau)$-modules: one corresponds to $\Lambda=\alpha+1-2\ZN_{>0}$, the other corresponds to $\beta+1+2\ZN_{\geq 0}$. In case $\beta-\alpha<2$ there are unitarizable irreducible modules that correspond to arithmetic progressions $\Lambda\subset \RN/2\ZN$ that do not intersect $[\alpha+1,\beta+1]$. Finite-dimensional module is unitarizable only when it is one-dimensional.
\end{enumerate}
\end{enumerate}  

A classical construction of \HC connects unitarizable irreducible \HC modules and unitary representations of groups $\SL(2,\RN)$ and $\SU(2,\CN)$. More precisely, there is a one-to-one correspondence between 
\begin{enumerate}
\item
Unitarizable irreducible \HC modules with integral weights for all $P=(\lambda+1)^2-x^2$, where $(\lambda+1)^2\in \RN$.
\item
Unitary representations of $\SL(2,\RN)$.
\end{enumerate}
Similarly for $P=(\lambda+1)^2+x^2$ and unitary representations $\SU(2,\CN)$.

Looking at classification of irreducible \HC modules for $P$ with positive leading coefficient we recover classification of irreducible unitary representations of $\SU(2,\CN)$: there exists a unique irreducible unitary representations of given finite dimension.

Looking at classification of irreducible \HC modules for $P$ with negative leading coefficient we recover Bargmann classification. When $(\lambda+1)^2<0$ $P$ has two different complex conjugate roots and we get two integrable unitarizable modules: one with even weights and one with odd weights. They correspond to principal series representations of $\SL(2,\CN)$. When $\lambda=-1$ $P$ has a root with multiplicity two and we get three integrable unitarizable modules. One of them corresponds to principal series representation of $\SL(2,\CN)$, other two correspond to limit of discrete series representations.  When $-1<\lambda<0$ then $P$ has two roots $\alpha<\beta<\alpha+2$ and we get an integrable unitarizable module corresponding to complementary series representation. When $\lambda\in \ZN$ we have $\alpha,\beta\in \ZN$ and we get two integrable unitarizalbe modules corresponding to discrete series representations. The trivial module corresponds to the trivial representation.
\section{Harish-Chandra bimodules}
\label{SecHCBimodules}
\subsection{Case of $\mf{sl}_2$}
\label{SubSecSL2}
\paragraph*{Connection to unitary representations of $\SL(2,\CN)$.}
We write the classification of irreducible \HC bimodules in case $n=2$ for reader's convenience.

Consider an infinite-dimensional irreducible unitary representation $V$ of $\SL(2,\CN)$ considered as a real group. Such representations are in one-to-one correspondence with irreducible unitary $(\mf{sl}(2,\CN),\operatorname{SU}(2))$-modules $V$, in other words $\mf{sl}(2,\CN)$-modules such that $\mf{su}_2$ acts locally finitely with a positive definite Hermitian form $(\cdot,\cdot)$ such that for any $x\in \mf{sl}(2,\CN), u,v\in V$ we have $(xu,v)+(u,xv)=0$.

The lie algebra $\mf{sl}_2$ is a real form of $\mf{sl}_2\times\mf{sl}_2^{op}$ corresponding to the antilinear involution $(a,b)\mapsto (b^*,a^*)$. This allows us to introduce a $\CN$-linear action of $\mf{sl}_2\times \mf{sl}_2^{op}$ on $V$:

\[(a,b).v=\frac{1}{2}((a+b^*,b+a^*)+(a-b^*,b-a^*)).v=\frac{1}{2}((a+b^*)v-i(ia-ib^*)v).\]

We see that the elements $(a,-a)$ act locally finitely. The condition on $(\cdot,\cdot)$ is as follows:

\begin{multline*}
2((a,b).u,v)=((a+b^*)u,v)-i((ia-ib^*)u,v)=\\
-(u,(a+b^*)v)-i(u,(ia-ib^*)v)=(u,(-a-b^*)v+i(ia-ib^*)v)=2(u,(-b^*,-a^*)v).
\end{multline*}

It follows that $V$ is an $\mf{sl}_2$-bimodule such that the adjoint action is locally finite and $(au,v)=-(u,va^*)$, $(ub,v)=-(u,b^*v)$. Using the example~\ref{ExClassicalAndATau} from the introduction we see that this is the same as a Harish-Chandra $U(\mf{sl}_2)$-bimodule. Rewrite the conditions on $(\cdot,\cdot)$ in terms of generators:

\begin{align*}
(eu,v)=-(u,vf) \quad &
(ue,v)=-(u,fv)\\
(hu,v)=-(u,vh) \quad &
(uh,v)=-(u,hv)\\
(fu,v)=-(u,ve) \quad &
(uf,v)=-(u,ev).\\
\end{align*}

It is easy to see that this is equivalent to the following:

\begin{align}
([e,u],v)=(u,[f,v])\quad &
([h,u],v)=(u,[h,v])\quad &
([f,u],v)=(u,[e,v])\\
(eu,v)=(u,-vf)
\end{align}

This gives a motivation for our choice of $\tau\colon\mc{A}\to \mc{A}$: $\tau(e)=-f$, $\tau(f)=-e$, $\tau(h)=-h$.

Suppose that $\mc{A}$ is a noncommutative deformation of $\CN[x,y]^{C_2}$, this is the same as a central reduction of $U(\mf{sl}_2)$: $\mc{A}=U(\mf{sl}_2)/(ef+fe+\frac{h^2}{2}-\frac{\lambda^2}2-\lambda)$. We assume that $-\frac{\lambda^2}2-\lambda\in\RN$, this is the same as $(\lambda+1)^2\in \RN$. In this subsection we classify unitarizable irreducible Harish-Chandra $\mc{A}$-bimodules.

Classification of irreducible \HC $\mc{A}$-bimodules is a known result~\cite{BG}:
\begin{enumerate}
\item
$\lambda\in\ZN\setminus \{-1\}$. Since $\lambda$ and $-2-\lambda$ give the same deformation we can assume that $\lambda\geq 0$. In this case $\mc{A}$ has a finite-dimensional representation $V$, $\dim V=\lambda+1$. Denote by $I$ the annihilator of $V$. There are two irreducible Harish-Chandra bimodules: $I$, $\End(V)$.
\item
$\lambda\in\ZN+\frac{1}{2}$. In the case $\lambda=-\frac{1}{2}$ the algebra $\mc{A}$ is an invariant subalgebra of the Weyl algebra under the natural action of $C_2$. We have two irreducible Harish-Chandra bimodules: $\mc{A}$ and the second isotypic component of the action of $C_2$. In the other cases we get an equivalent category, so it also contains two irreducible bimodules.
\item
In all other cases there is one irreducible Harish-Chandra bimodule, $\mc{A}$.
\end{enumerate}

First we consider the relations

$$
([e,u],v)=(u,[f,v]);\quad 
([h,u],v)=(u,[h,v]);\quad 
([f,u],v)=(u,[e,v]).
$$

Since the adjoint action of $\mf{sl}_2$ on $V$ is locally finite, $V$ is a direct sum of irreducible $\mf{sl}_2$-modules. It is not hard to see that for an irreducible bimodule $V$ all irreducible $\mf{sl}_2$-modules are distinct.

Suppose that $V_k$ is an irreducible $\mf{sl}_2$-submodule of highest weight $k$, $U_k$ is obtained from $V_k$ using the automorphism of $\mf{sl}_2$ that sends $e,f,h$ to $-f,-e,-h$ respectively. Then $(\cdot,\cdot)$ gives an invariant sesquilinear pairing between $V_k$ and $U_k$. Since $U_k$, $V_k$ are isomorphic irreducible modules, this form is unique up to a scalar. There exists a positive definite invariant Hermitian form on $V_1$, hence there exists a positive definite invariant Hermitian form on $V_k=S^k V_1$ for all $k\geq 0$.

A similar argument shows that $V_k$ is orthogonal to $U_l$ when $k\neq l$.

We are left with the condition

\begin{equation}
\label{EqInvarEq1}
(eu,v)=(u,-vf).
\end{equation} 

It follows from Proposition~\ref{StatUniqueHermitianForm2} below that there exists a unique up to a scalar invariant Hermitian form on $V$. It remains to check when this form is positive.


\begin{thr}
The following bimodules are unitarizable:
\begin{enumerate}
\item
$\CN$ for $\lambda=0$.
\item
The regular bimodule for $\lambda$ such that $(\lambda+1)^2<1$.
\item
The annihilator of the finite-dimensional representation in the case $\lambda\in \ZN_{\geq 0}$.
\item
Non regular irreducible bimodule in the case $\lambda\in \frac{1}{2}+\ZN$.
\end{enumerate}
\end{thr}
\begin{proof}
Denote an irreducible bimodule by $M$.

To check that the form $(\cdot,\cdot)$ is positive it is enough to check that $(\cdot,\cdot)|_{V_l}$ and $(\cdot,\cdot)_{V_{l+2}}$ have the same sign. We are going to find an equality of the form $(eu,eu)=a(u,u)$. Let $u$ be a highest weight vector of $V_l$. If $eu=0$ then it is not hard to see that $\oplus_{m\leq l} V_m$ is a subrepresentation. Since $V$ is irreducible $V_{l+2}$ does not occur in $V$. So we can assume that $eu\neq 0$. Hence it is a highest weight vector of $V_{l+2}$. Denote $l=2k$.

\begin{lem*}
We have \[(eu,eu)=\frac{k+1}{4k+6}((k+1)^2-(\lambda+1)^2)(u,u).\]
\end{lem*}
\begin{proof}

Let us compute projections of $hu$ on $V_l$, $V_{l+2}$. This is equivalent to expressing $hu$ as a linear combination of $u$ and $[f,eu]$.

We have $[f,eu]=-hu+e[f,u]=-hu+[ef,u]$. The term $[ef,u]$ equals to

\[[ef,u]=[-\frac{h^2}{4}+\frac{h}{2},u]=k(-\frac{1}{2}(hu+uh)+u)=k(-hu+(k+1)u).\]

Hence $[f,eu]=-(k+1)hu+k(k+1)u$. It follows that $hu=-\frac{1}{k+1}[f,eu]+ku$. In particular,
\begin{equation}
\label{EqProductuHu}
(u,hu)=k(u,u).
\end{equation}

We have $[e,[f,eu]]=[[e,f],eu]]+[f,[e,eu]]=[h,eu]=(2k+2)eu$.

Hence \begin{equation}
\label{EqHuOneWay}
(hu,[f,eu])=-\frac{1}{k+1}([f,eu],[f,eu])=-\frac{1}{k+1}([e,[f,eu]],eu)=-2(eu,eu).
\end{equation} 

On the other hand 
\begin{multline}
\label{EqHuBeginningOtherWay}
(hu,[f,eu])=-(u,h[f,eu])=-(u,h(-(k+1)hu+k(k+1)u))=
(u,(k+1)h^2u-k(k+1)hu).
\end{multline}
We have $2fe=-\frac{h^2}{2}-h+\frac{\lambda^2}{2}+\lambda$. It follows that \[(k+1)h^2-k(k+1)h=-4(k+1)fe-(k+2)(k+1)h+(k+1)(\lambda^2+2\lambda).\]

Continuing~\eqref{EqHuBeginningOtherWay} and using~\eqref{EqProductuHu} we have \begin{multline*}
(u,(k+1)h^2u-k(k+1)u)=(u,-4(k+1)fe-(k+2)(k+1)hu+(k+1)(\lambda^2+2\lambda)u)=\\
-4(k+1)(u,feu)-(k+1)(k+2)(u,hu)+(k+1)(\lambda^2+2\lambda)(u,u)=\\
4(k+1)(ue,eu)-(k+1)(k+2)k(u,u)+(k+1)(\lambda^2+2\lambda)(u,u)=\\
4(k+1)(eu,eu)+(k+1)((\lambda+1)^2-(k+1)^2)(u,u).
\end{multline*}

It follows that \[(hu,[f,eu])=4(k+1)(eu,eu)+(k+1)((\lambda+1)^2-(k+1)^2)(u,u).\] Comparing with~\eqref{EqHuOneWay} we have
\[-2(eu,eu)=4(k+1)(eu,eu)+(k+1)((\lambda+1)^2-(k+1)^2)(u,u).\] Hence \[(eu,eu)=\frac{k+1}{4k+6}((k+1)^2-(\lambda+1)^2)(u,u).\]

\end{proof}
Let $V_l\subset M$ be an irreducible $\mf{sl}_2$-submodule with the minimal highest weight. There are two cases:
\begin{enumerate}
\item
$eu=0$. As we saw in this case $M=\oplus_{m\leq l}V_l$, hence $M=V_l$. It is easy to deduce that $M=\CN$. This is a unitarizable bimodule.
\item
$eu\neq 0$. Let $2k=l$. If $M$ is unitarizable then $(k+1)^2-(\lambda+1)^2>0$. On the other hand if $(k+1)^2-(\lambda+1)^2>0$ then for all $V_m\subset M$ we have $m\geq l$, $(k_1+1)^2-(\lambda+1)^2\geq (k+1)^2-(\lambda+1)^2>0$ where $2k_1=m$. 

In case $\lambda\notin\RN$ we deduce that $M$ is unitarizable. Suppose that $\lambda\in \RN$.

Assume that $\lambda\geq -1$. We deduce that $M$ is unitarizable if and only if for the minimal highest weight $l$ we have $l>2\lambda$. It is not hard to see that for the annihilator of the finite-dimensional representation in case $\lambda\in\ZN_{\geq 0}$ and the second bimodule in case $\lambda\in \frac{1}{2}+\ZN$ we have $l=2\lambda+2$.

There remain two cases: $M=\mc{A}$ and $M$ finite dimensional. In these cases $l=0$. So unitarizability is equivalent to $-1\leq \lambda<0$. For these $\lambda$ $\mc{A}$ is irreducible and has no finite-dimensional representations.
\end{enumerate} 
\end{proof}
These unitarizable bimodules correspond to irreducible unitary representations of $\SL(2,\CN)$ as follows. The regular bimodule for $\lambda=-1$, the annihilator of the finite-dimensional representation for $\lambda\in \ZN_{\geq 0}$ and non-regular irreducible bimodule in the case $\lambda\in \frac{1}{2}+\ZN$ correspond to the principal series representations. The regular bimodule for $-1<\lambda<0$ corresponds to the complementary series representations. See~\cite{Ta}, for example.

Other unitary representations correspond to $\mc{A}_{\lambda}$-$\mc{A}_{\ovl{\lambda}}$ bimodules for $(\lambda+1)^2$ complex.

\subsubsection{The case of the Weyl algebra}
We will need the definition of a twisted trace.
\begin{defn}
Suppose that $A$ is an algebra over $\CN$, $g\colon A\to A$ is an automorphism. We say that a linear map $T\colon A\to \CN$ is a $g$-twisted trace if $T(ab)=T(bg(a))$ for all $a,b\in A$.
\end{defn}
Let $s\colon\mc{A}\to\mc{A}$ be the map $sx=-x$, $sy=-y$.

The Weyl algebra $\mc{A}=\CN\langle x,y\rangle/(xy-yx-1)$ has no nonzero traces, but has a nonzero $s$-twisted trace. Consider the antilinear automorphism $r\colon \mc{A}\to\mc{A}: x\mapsto y$, $y\mapsto -x$. We have $r^2=s$. We get an antilinear involution on $\mc{A}\otimes \mc{A}$: $a\otimes b\mapsto r^{-1}(b)\otimes r(a)$. So we can ask if $\mc{A}$ has an invariant positive definite form.

We see that $\mc{A}^{C_2}$ is generated by $e=\frac12x^2, h=-\frac{xy+yx}2, f=-\frac12y^2$. This is a deformation of $\CN[x,y]^{C_2}$ with parameter $\lambda=-\frac{1}{2}$. We have $re=-f$, $rf=-e$, $rh=-h$. This is the real form that we considered in the case of $\mf{sl}_2$. Using the results for this case we see that both $\mc{A}^{C_2}$ and the second isotypic component $M$ have an $\mc{A}^{C_2}$-invariant positive definite form. We get a form $(\cdot,\cdot)$ on $\mc{A}$ such that $(x,x)=\frac{1}{2}(1,1)$, $M\perp\mc{A}^{C_2}$, $(\cdot,\cdot)$ restricted to $M$ or $\mc{A}^{C_2}$ is an $\mc{A}^{C_2}$-invariant positive definite form.

It remains to check that $(au,v)=(u,vr(a))$, $(ua,v)=(u,r^{-1}(a)v)$ for all $a,u,v\in\mc{A}$.

Denote by $T$ the $s$-twisted trace on $\mc{A}$ such that $T(1)=1$. It is enough to check that $(u,v)=T(ur(v))$ for all $u,v\in \mc{A}$. For $u$, $v$ in different isotypic components this is clear. For $u,v\in \mc{A}^{C_2}$ we have $(u,v)=(ur(v),1)=T'(ur(v))$, where $T'$ is an invariant trace on $\mc{A}^{C_2}$ such that $T'(1)=1$. Such trace is unique, hence $T'=T|_{\mc{A}^{C_2}}$. It follows that $(u,v)=T(ur(v))$. 

For $u,v\in M$ we have $u=ax$ or $u=ay$. Assume that $u=ax$. Hence $(u,v)=(ax,v)=(x,vr(a))$. We have $vr(a)=\alpha x+t$, where $\alpha\in \CN$, $t\in \CN y+\oplus_{k=1}^{\infty} V_{2k+1}$, $V_i$ is an irreducible $\mf{sl}_2$-module of highest weight $i$. Since both $(\cdot,\cdot)$ and $T$ are $\mf{sl}_2$-invariant we get $(x,vr(a))=\frac{\alpha}{2}$, $T(ur(v))=T(axr(v))=T(\alpha xy)=\frac{\alpha}{2}$.

Hence $(\cdot,\cdot)$ is $\mc{A}$-invariant and positive definite.
\subsection{Invariant Hermitian forms on the regular bimodule}
In $\mc{A}$ we have $[h,e]=2e$, $[h,f]=-2f$, $ef=P(h-1)$, $fe=P(h+1)$.

In the case $P(x)\in\RN[ix]$ we have an antilinear involution $e\mapsto -f$, $f\mapsto -e$, $h\mapsto -h$. Denote by $\ovl{a}$ the image of $a\in\mc{A}$ under this antilinear involution. The form $(\cdot,\cdot)$ is invariant if and only if

\begin{align*}
(eu,v)=-(u,vf) \quad &
(ue,v)=-(u,fv)\\
(hu,v)=-(u,vh) \quad &
(uh,v)=-(u,hv)\\
(fu,v)=-(u,ve) \quad &
(uf,v)=-(u,ev)
\end{align*}
for all $u,v\in\mc{A}$.

Since $(\cdot,\cdot)$ is Hermitian it is enough to consider the first three equalities:

\begin{align*}
(eu,v)=-(u,vf)\\
(hu,v)=-(u,vh)\\
(fu,v)=-(u,ve)
\end{align*}

Since $(\cdot,\cdot)$ is invariant with respect to $\ad h$ we see that weight spaces of $\ad h$ in $\mc{A}$ of different weight are orthogonal with respect to $(\cdot,\cdot)$.

We have $\mc{A}_{2k}=\{S(h)e^k\mid S(x)\in\CN[x]\}$ for $k\geq 0$ and $\mc{A}_{2k}=\{S(h)f^k\mid S(x)\in\CN[x]\}$ for $k\leq 0$.

It is obvious that $S(h)e^k=e^kS(h+2k)$, $S(h)f^k=f^kS(h-2k)$.


The following proposition is proved in a straightforward way.
\begin{prop*}
Invariant Hermitian forms on $\mc{A}$ are in one-to-one correspondence with traces on $\mc{A}$. The correspondence is as follows: from $(\cdot,\cdot)$ we get $T(a)=(a,1)$, from $T$ we get $(a,b)=T(a\ovl{b})$.
\end{prop*}

Let us classify traces on $\mc{A}$. It is enough to check the condition $T(ab)=T(ba)$ when $a$ equals to $e,f$ or $h$. The equality $T(hb)=T(bh)$ says that $T$ is supported on $\mc{A}_0$. 

It follows that is enough to check the equality $T(eb)=T(be)$ for $b\in \mc{A}_{-2}$. In this case $b=fS(h-1)$. We have \[T(eb)=T(efS(h-1))=T(P(h-1)S(h-1)),\] \[T(be)=T(fS(h-1)e)=T(fe S(h+1))=T(P(h+1)S(h+1)).\]

Similarly it is enough to check the equality $T(fb)=T(bf)$ for $b\in \mc{A}_2$. In this case $b=e S(h+1)$. We have $T(fb)=T(fe S(h+1))=T(P(h+1)S(h+1))$. We also have $T(bf)=T(e S(h+1)f)=T(ef S(h-1))=T(P(h-1)S(h-1))$.

So we have proved the following proposition:

\begin{prop}
$T\colon\mc{A}\to \CN$ is a trace on $\mc{A}$ if and only if
\begin{enumerate}
\item
$T$ is supported on $\mc{A}_0$.
\item
$T(S(h-1)P(h-1))=T(S(h+1)P(h+1))$ for all $S \in\CN[x]$.
\end{enumerate}
\end{prop}
Therefore the space of traces \[(\CN[x]/\{S(x+1)P(x+1)-S(x-1)P(x-1)\mid S \in\CN[x]\})^*\] has dimension $n-1$.

\begin{rem}
If $g$ is a filtration-preserving automorphism of $\mc{A}$ then we can consider a $g$-twisted trace: linear map $T\colon \mc{A}\to \CN$ such that $T(ab)=T(bg(a))$ for all $a,b\in \mc{A}$. We plan to discuss $g$-twisted traces on $\mc{A}$ in more details in the joint paper with Pavel Etingof, Douglas Stryker and Eric Rains. 
\end{rem}

Let $\mc{A}_{\RN}=\{a\in\mc{A}\mid a=\ovl{a}\}$.  We have $\mc{A}_{\RN}\cap \mc{A}_0=\RN[ih]$. It is easy to prove that $(\cdot,\cdot)$ is Hermitian if and only if $T$ is real on $\RN[ih]$.

From now on we suppose that $T$ is real on $\RN[ih]$.

Suppose that $a\in\RN$. Denote $\Re_a S(x)=\frac{1}{2}(S(x)+\ovl{S}(2a-x))$, $\Im_a S(x)=\frac{1}{2i}(S(x)-\ovl{S}(2a-x))$. The polynomial $\Re_a S(x)$ is the unique polynomial such that $\Re_a S(x)=\Re(S(x))$ when $\Im x=a$. Similarly, $\Im_a S(x)$ is the unique polynomial such that $\Im_a S(x)=\Im(S(x))$ when $\Im x=a$. It follows that $P(x)=\Re_0 P(x)$.
\begin{lem}
\label{LemSpaceOnWhichTraceActs}
Suppose that $\phi\colon\CN[x]\to \CN$ is a linear map such that $\phi|_{\RN[ix]}$ is real. Then $\phi$ is zero on $\{S(x+1)P(x+1)-S(x-1)P(x-1)\mid S(x)\in\CN[x]\}$ if and only if $\phi$ is zero on $\{S(x+1)P(x+1)-S(x-1)P(x-1)\mid \Re_{0}S(x)=0\}$. Moreover, when $\Re_{0}S(x)=0$ we have $S(x+1)P(x+1)-S(x-1)P(x-1)=-2\Re_0 (S(x-1)P(x-1))$.
\end{lem}
\begin{proof}
Let $M$ be any subset of $\CN[x]$. Since $\phi|_{\RN[ix]}$ is real, $\phi(M)=\{0\}$ if and only if $\phi(\Re_0 M\cup \Im_0 M)=\{0\}$. It is  easy to see that for \[M=\{S(x+1)P(x+1)-S(x-1)P(x-1)\mid S(x)\in \CN[x])\] we have \[\Re_0 M\cup \Im_0 M=\{S(x+1)P(x+1)-S(x-1)P(x-1)\mid \Re_{0}S(x)=0\}.\]
The second statement is straightforward.
\end{proof}
Denote by $L$ the space $\RN[ix]/\{\Re_0(S(x-1)P(x-1))\mid \Re_{0}S(x)=0\}$.

\begin{cor*}
Invariant Hermitian forms on $\mc{A}$ are in one-to-one correspondence with linear maps $f\colon L\to\RN$.
\end{cor*}

Denote by $C_0$ the convex cone in $\RN[ix]$ generated by $a\ovl{a}\in \mc{A}_0\cap \mc{A}_{\Re}$, where $a$ is a nonzero homogeneous element of $\mc{A}$. Denote by $C$ the image of $C_0$ in $L$. The following proposition is straightforward.
\begin{prop*}
The form $(\cdot,\cdot)$ is positive definite if and only if $f(C)\subset \RN_{>0}$. The form $(\cdot,\cdot)$ is positive semidefinite if $f(C)\subset \RN_{\geq 0}$. 
\end{prop*}

\begin{cor}
\label{CorConnectionCandUnitary}
\begin{enumerate}
\item
There exists a nonzero  positive semidefinite invariant form on $\mc{A}$ if and only if $0\in\partial C$.
\item
If $0\in C$ then $\mc{A}$ is not unitarizable.
\item
If $0\notin C$ then there exists an invariant positive semidefinite form on $\mc{A}$.
\end{enumerate}
\end{cor}
\begin{proof}
The first statement follows from the supporting hyperplane theorem. The third statement follows. The second statement is straightforward.
\end{proof}

\begin{prop}
\label{StatHowCLooks}
\begin{enumerate}
\item
For any $a\in\mc{A}_{2k}$, $k\in\ZN$ there exists $b\in\mc{A}_0\cup\mc{A}_2$ such that for any invariant form $(\cdot,\cdot)$ on $\mc{A}$ we have $(a,a)=(b,b)$.
\item
$C$ is the image of \[\Bigl\{\ovl{R_1}(-x)R_1(x)-\Re_0 \ovl{R}_2(1-x)R_2(x-1)P(x-1)\mid (R_1,R_2)\in \CN[x]^2\setminus\{(0,0)\}\Bigr\}\] in $L$.
\end{enumerate}
\end{prop}
\begin{proof}
\begin{enumerate}
\item
Suppose that $a$ has weight $4k>0$. Then $a=e^kS(h)e^k$ for some $S \in\CN[x]$. Hence
\[(a,a)=\left(e^k S(h) e^k, e^k S(h) e^k\right)=\Bigl(S(h)e^k (-f)^k, S(h)e^k (-f)^k\Bigr).\]

Suppose that $a$ has weight $4k+2$, $k\geq 0$. Then $a=e^k S(h)e^{k+1}$ for some $S(x)\in\CN[x]$. Hence

\[(a,a)=\left(e^k S(h) e^{k+1},e^k S(h) e^{k+1}\right)=\Bigl(S(h)e^{k+1}(-f)^k,S(h)e^{k+1}(-f)^k\Bigr).\]
The case of negative degree is done similarly.
\item
By definition $C_0$ is generated by images of $a\ovl{a}$ in $L$ for all nonzero homogeneous $a\in \mc{A}$. We take $b\in \mc{A}_0 \cup \mc{A}_2$ such that $(a,a)=(b,b)$ for all invariant forms. This means that the images of $a\ovl{a}$ and $b\ovl{b}$ in $L$ coincide. It follows that $C$ is generated by the images of $a\ovl{a}$ in $L$ for all $a\in\mc{A}_0\cup\mc{A}_2$.

Take $a=R_1(h)$.  We have \[(a,a)=(R_1(h),R_1(h))=(R_1(h)\ovl{R_1}(-h),1).\]
Take $a=R_2(h-1)e\in\mc{A}_2$. We have
\begin{multline*}
(a,a)=(R_2(h-1)e,R_2(h-1)e)=\\
-\frac{1}{2}\Bigl(\ovl{R_2}(-1-h)fR_2(h-1)e,1\Bigr)+\Bigl(R_2(h-1)e\ovl{R_2}(-1-h)f,1\Bigr)=\\
-\frac{1}{2}\Bigl(\ovl{R_2}(-h-1)fe R_2(h-1)+R_2(h-1)ef\ovl{R_2}(-h+1),1\Bigr)=\\
-\Re_0\Bigl(\ovl{R_2}(1-h)R_2(h-1)P(h-1)),1\Bigr).
\end{multline*}

It follows that the image of $a\ovl{a}$ in $L$ equals to the image of $R_1(h)\ovl{R_1}(-h)$ or $-\Re_0\left(\ovl{R_2}(1-h)R_2(h-1)P(h-1)\right)$ in $L$. It is easy to see that a polynomial $S(x)\neq 0$ equals to $R_1(x)\ovl{R_1}(-x)$ for some $R_1\in\CN[x]$ if and only if $S|_{i\RN}\geq 0$. Similarly $S(x)$ can be represented as $R_2(x-1)\ovl{R_2}(1-x)$ if and only if $S|_{i\RN+1}\geq 0$. Hence \[\left\{\ovl{R_1}(-x)R_1(x)-\Re_0 \ovl{R}_2(1-x)R_2(x-1)P(x-1)\mid R_1,R_2 \in\CN[x]^2\setminus\{(0,0)\}\right\}\] is a cone, so its image coincides with $C$.
\end{enumerate}
\end{proof}

Now we are able to prove the following
 
\begin{prop}
\label{StatWhenCcontainsZero}
$C$ contains zero if and only if there exists nonzero $F \in\CN[x]$ such that
\begin{enumerate}
\item
$\Re F(x)\geq 0$ when $\Re x=0$.
\item
$\Re F(x-1)P(x-1)\geq 0$ when $\Re x=0$.
\end{enumerate}
\end{prop}
\begin{proof}
It follows from Proposition~\ref{StatHowCLooks} that $C$ does not contain zero if and only if there do not exist $R_1,R_2,S\in \CN[x]$ not all equal to zero such that $\Re_{0}S(x)=0$ and \[\ovl{R_1}(-x)R_1(x)-\Re_0 \ovl{R_2}(1-x)R_2(x-1)P(x-1)=\Re_0 (S(x-1)P(x-1)).\] This can be rewritten as \[\ovl{R_1}(-x)R_1(x)=\Re_0\bigl((S(x-1)+\ovl{R_2}(1-x)R_2(x-1))P(x-1)\bigr).\]

 We see that set $\left\{\ovl{R_1}(-x)R_1(x)\mid R_1(x)\in\CN[x]\right\}$ coincides with $\left\{Q(x)\mid Q|_{i\RN}\geq 0\right\}$. For any $F(x)\in \CN[x]$ we have $(\Re_0 F)(x)=\Re(F(x))$ when $x\in i\RN$. Therefore such $R_1$ exists if and only if \[\Re\Bigl(\left(S(x-1)+\ovl{R_2}(1-x)R_2(x-1)\right)P(x-1)\Bigr)\geq 0\] when $x\in i\RN$.
 
Denote $S(x)+\ovl{R_2}(-x)R_2(x)$ by $F(x)$.  We see that the set \[\Bigl\{S(x)+\ovl{R}(-x)R(x)\mid S(x),R(x)\in \CN[x], \Re_0 S(x)=0\Bigr\}\] coincides with \[\left\{Q(x)\mid \Re (Q(x)|_{i\RN})\geq 0\right \}.\] The statement follows.
\end{proof}
\subsection{Proof of main theorem.}
Denote by  $\rho_{< a}(F(x))$ the number of roots of $F(x)$ with real part less than $a$, similarly for other inequality signs.
\begin{cor}
\label{CorCdoesNotContainZero}
Suppose that one of the following holds:
\begin{enumerate}
\item
 $P(x)$ has at least three roots $\alpha$ with multiplicities with $|\Re\alpha|<1$.
\item
 $P(x)$ has two roots $\alpha$ with multiplicities with $|\Re\alpha|<1$. Denote the degree of $P(x)$ by $2m$. The leading  coefficient of $P(x)$ has sign $(-1)^m$.
\end{enumerate}
Then $C$ does not contain zero.
\end{cor}
\begin{proof}
The number of roots of $P(x)$ with multiplicites that satisfy $|\Re\alpha|<1$ equals to $\rho_{>-1}(P)-\rho_{\geq 1}(P)$. Since $P(x)=\ovl{P}(-x)$ we have $\rho_{\geq 1}(P)=\rho_{\leq -1}(P)$. Therefore the number of roots of $P(x)$ that satisfy $|\Re\alpha|<1$ equals to $\rho_{>-1}(P)-\rho_{\leq -1}(P)$. So in the first case we have $\rho_{>-1}(P(x))\geq \rho_{\leq -1}(P(x))+3$ and in the second case we have $\rho_{>-1}(P(x))\geq \rho_{\leq -1}(P(x))+2$.

If $C$ contains zero then using  Proposition~\ref{StatWhenCcontainsZero} we get $F(x)$ such that $\Re F(x)\geq 0$ when $\Re x=0$, $\Re F(x)P(x)\geq 0$ when $\Re x=-1$. Using Lemma~\ref{LemWindingNumber} in Appendix we see that \[\rho_{<0}(F(x))\leq \rho_{\geq 0}(F(x))+1,\] \[\rho_{>-1}(F(x)P(x))\leq \rho_{\leq -1}(F(x)P(x))+1.\] It follows that \[\rho_{\leq -1}(F(x))\leq \rho_{>-1}(F(x))+1$$ $$\rho_{>-1}(F(x))+\rho_{>-1}(P(x))\leq \rho_{\leq -1}(F(x))+\rho_{\leq -1}(P(x))+1.\] Adding these two inequalities we get \[\rho_{>-1}P(x)\leq \rho_{\leq -1}(P(x))+2.\] In the first case of the corollary we get a contradiction.

Consider the second case. We see that the inequalities on $\rho$ become equalities. Using Lemma~\ref{LemWindingNumber} again we deduce that $F(x)$ has degree $2d-1$, the leading coefficient of $F(x)$ has sign $(-1)^{d-1}$, the leading coefficient of $F(x)P(x)$ has sign $(-1)^{m+d}$. It follows that the leading coefficient of $P(x)$ has sign $(-1)^{m+1}$, a contradiction. 
\end{proof}
Using Corollary~\ref{CorConnectionCandUnitary} we deduce that under these conditions on $P$ the algebra $\mc{A}$ has a positive semidefinite Hermitian form. Hence either $\mc{A}$ is unitarizable or there exists an ideal $I$ in $\mc{A}$ such that $\mc{A}/I$ is unitarizable. This motivates the following two propositions.

\begin{prop}
Every nonzero ideal of $\mc{A}$ has finite codimension.
\end{prop}
\begin{proof}
It is enough to prove that every Poisson ideal $I$ of $\CN[x,y]^{C_n}$ has finite codimension. $I$ corresponds to a closed Poisson subscheme $Y$ of $X=\CN^2/C_n$. $X$ has two symplectic leaves: $\{0\}$ and $X\setminus \{0\}$. It follows that $Y$ is supported on $\{0\}$, hence $I$ has finite codimension.
\end{proof}
\begin{prop}
Suppose that $V$ is an irreducible finite-dimensional unitarizable $\mc{A}$-bimodule. Then $\dim V=1$ and there exists $\lambda\in\CN$ such that $hv=\lambda v$, $vh=-\ovl{\lambda}v$, $P(\lambda\pm 1)=0$.
\end{prop}
\begin{proof}
Using the double centralizer theorem we see that $V=U\otimes W^*$, where $U,W$ are irreducible $\mc{A}$-modules. The operator $h$ has an eigenvector on a finite-dimensional vector space. Reasoning as in subsection~\ref{SubSecIrredModules} we see that $h$ acts on $U,W$ diagonalizably with one-dimensional eigenspaces. Denote by $S_U$, $S_W$ the sets of eigenvalues of $h$ on $U,W$. Hence there is a bigrading on $V$ by left and right action of $h$ and the set of weights equals to $S_U\times S_W$.

Suppose that $v$ is a nonzero homogeneous element of weight $(\lambda,\mu)$. Since $(\cdot,\cdot)$ is invariant we have $(hv,v)=-(v,vh)$. Since $(v,v)>0$ we have $\lambda=-\ovl{\mu}$. It follows that for any element of $S_U$ there exists exactly one element of $S_W$ and vice versa. Hence $|S_U|=|S_W|=1$.

So there exists $\lambda\in\CN$ such that $hv=\lambda v$, $vh=-\ovl{\lambda}v$. Since $efv=fev=0$ we get $P(\lambda+1)=P(\lambda-1)=0$.
\end{proof}

\begin{lem}
\label{LemConeInjection}
Suppose that $P_1$ and $P_2$ are two polynomials in $\RN[ix]$ such that $\frac{P_1}{P_2}$ is a polynomial that is nonnegative on the set $\Re x=0$. Denote by $\mc{A}_1$ and $\mc{A}_2$ the corresponding algebras. Then there is a natural injective map from the cone of positive definite forms on $\mc{A}_2$ to the cone of positive definite forms on $\mc{A}_1$. In particular if the regular bimodule corresponding to $P_2$ is unitarizable then the regular bimodule corresponding to $P_1$ is also unitarizable.
\end{lem}
\begin{proof}
For $P(x)\in \RN[ix]$ of degree $n$ denote by $L(P)$ the dual space to the space of real traces on the deformation of $\CN[x,y]^{C_n}$ that corresponds to $P$. We have \[L(P)=\RN[ix]/\{P(x+1)S(x+1)-P(x-1)S(x-1)\mid S(x)\in i\RN[ix]\}.\] Denote by $C(P)$ the set of images of $a\ovl{a}$, $a\in\mc{A}$ in $L(P)$. From Proposition~\ref{StatHowCLooks} we get that $C(P)$ is the image of \[\Bigl\{\ovl{R_1}(-x)R_1(x)-\Re_0 \ovl{R}_2(1-x)R_2(x-1)P(x-1)\mid (R_1,R_2)\in \CN[x]^2\setminus\{(0,0)\}\Bigr\}\]  in $L(P)$.

It is easy to see that in this case we have a natural map $\phi\colon L(P_1)\to L(P_2)$ and $\phi(C(P_1))\subset \phi(C(P_2))$. Suppose that a positive definite form on $\mc{A}_2$ is given by the trace $\psi_2\colon L(P_2)\to \RN$. Then $\psi_2\circ \phi$ gives a trace on $\mc{A}_1$ and the corresponding form is positive definite.

\end{proof}
\begin{cor}
\label{CorUnitarizableFromC}
Suppose that the conditions of Corollary~\ref{CorCdoesNotContainZero} hold. Then $\mc{A}$ is unitarizable.
\end{cor}
\begin{proof}
Let $Q(x)$ be a polynomial that has all roots of $P(x)$ in the set $-1<\Re x<1$ with the same multiplicities. Then $\frac{P}{Q}$ is a polynomial that is real on $i\RN$ and has no roots on $i\RN$. Changing $Q$ to $-Q$ if necessary we can assume that $\frac{P}{Q}$ is positive on $i\RN$. Since $P(x)$ satisfies conditions of Corollary~\ref{CorCdoesNotContainZero} it is easy to see that $Q(x)$ satisfies conditions of Corollary~\ref{CorCdoesNotContainZero}. 

We deduce from Lemma~\ref{LemConeInjection} that we can change $P(x)$ to $Q(x)$. Denote the algebra that corresponds to $Q(x)$ by $\mc{A}$.

It is easy to see that there does not exist $\lambda$ such that $Q(\lambda+1)=Q(\lambda-1)=0$. It follows that $\mc{A}$ does not have irreducible unitarizable finite-dimensional modules. Since every ideal in $\mc{A}$ has a finite codimension we deduce that every positive semidefinite form on $\mc{A}$ is positive definite. Now we get a result from Corollaries~\ref{CorCdoesNotContainZero} and~\ref{CorConnectionCandUnitary}.
\end{proof}

\begin{rem}
Suppose that $P(x)$ is nonnegative on the line $\Re x=0$ and has a root $\alpha$ with $-1<\Re\alpha<1$. Then $P(x)$ is divisible by $P_2(x)=-(x-\alpha)(x+\ovl{\alpha})$ and $\frac{P}{P_2}$ is nonnegative on the line $\Re x=0$. Using results of subsection~\ref{SubSecSL2} we see that the regular bimodule corresponding to $P_2(x)$ is unitarizable. Using Lemma~\ref{LemConeInjection} we deduce that the regular bimodule corresponding to $P$ is unitarizable. Now there are four proofs that the regular bimodule with $P(x)$ nonnegative on $\Re x=0$ that has a root $\alpha$ with $-1<\Re\alpha<1$ is unitarizable: corollary~\ref{CorUnitarizableFromC}, this remark, analytic formula in subsection~\ref{SubSecAnalytic} and minimum principle in subsection~\ref{SubSubSecHarmonic} below. 
\end{rem}

\begin{rem}
\label{RemNonSemisimplicity}
Suppose that $P(x)=(-x^2)^m+\cdots$ has two roots $\alpha$ with $-1<\Re\alpha<1$ and $P(1)=P(-1)=0$. It follows that $\mc{A}$ has a one-dimensional representation generated by vector $v$ such that $ev=fv=hv=0$. Hence we have an exact sequence of Harish-Chanra bimodules $0\to \operatorname{Ann}(V)\to \mc{A}\to \End(V)\to 0$. It is easy to see that this sequence does not split. On the other hand $\mc{A}$ is unitarizable. Therefore there exists non-semisimple unitarizable Harish-Chandra bimodules.
\end{rem}

We also have the following result about non-unitarizability.
\begin{thr}
Suppose that $n=2m$ and $P(x)$ has leading coefficient $(-1)^m$. If for all roots $\alpha$ of $P(x)$ one has $|\Re\alpha|>1$ then the regular bimodule is not unitarizable.
\end{thr}
\begin{proof}
Using Proposition ~\ref{StatWhenCcontainsZero} we see that it is enough to prove that the cone $C$ contains zero. Now we get the result from Corollary~\ref{CorGoodApprox} and Proposition~\ref{PropFromGoodApproxToNonUnit} in appendix.
\end{proof}

Now Theorem~\ref{ThrMain} follows from this theorem and Corollary~\ref{CorUnitarizableFromC}.

\subsection{An analytic construction of a positive definite form}
\label{SubSecAnalytic}
I learned the following approach for constructing positive definite forms from Pavel Etingof. It appeared in his ongoing joint work with Eric Rains and Douglas Stryker. 

Let $P(x)=(x^2-\lambda^2)P_1(x)$, where $0\leq\lambda<1$ and $P_1(x)$ is nonnegative on the line $i\RN$. Consider \[w(x)=\ds{\frac{e^{\pi i x}}{(e^{\pi i x}+e^{\pi i\lambda})(e^{\pi i x}+e^{-\pi i\lambda})}}.\] This function has the following properties:
\begin{enumerate}
\item
$w(x)$ is $2$-periodic and decays exponentially when $x$ tends to $\pm i\infty$
\item
$P(x)w(x+1)$ is holomorphic on $\Re x\in [-1,1]$.
\item
$w$ is positive on the line $\Re x=0$
\item
$w$ is negative on the line $\Re x=1$.
\end{enumerate}

Let $T(F(x))=\int_{\RN} F(x)w(x) |dx|$.
\begin{lem}
\begin{enumerate}
\item
$T$ defines a trace on $\mc{A}$.
\item
$T$ is positive on polynomials $R(x)\ovl{R}(x)$.
\item
$T$ is positive on polynomials $-R(x-1)\ovl{R}(1-x)P(x-1)$.
\end{enumerate}
Hence $T$ gives an invariant positive definite form on $\mc{A}$.
\end{lem}
\begin{proof}
\begin{enumerate}
\item
We have
\begin{multline*}T(F(x+1)P(x+1))=-i\int_{i\RN}F(x+1)P(x+1)w(x)dx=\\
-i\int_{i\RN}F(x+1)P(x+1)w(x+2)dx=-i\int_{i\RN+2}F(x-1)P(x-1)w(x)dx.
\end{multline*}

Since $P(x)w(x+1)$ is holomorphic on $\Re x\in [-1,1]$ $P(x-1) w(x)$ is holomorphic on $\Re x\in [0,2]$, so we have \begin{multline*}-i\int_{i\RN+2}F(x-1)P(x-1)w(x)dx=-i\int_{i\RN}F(x-1)P(x-1)w(x)dx=\\
\int_{i\RN}F(x-1)P(x-1)w(x)|dx|= T(F(x-1)P(x-1)).
\end{multline*}
\item
The polynomial $R(x)\ovl{R}(-x)$ is positive almost everywhere on $i\RN$, $w(x)$ is positive on $i\RN$. Hence \[T(R(x)\ovl{R}(-x))=\int_{i\RN} R(x)\ovl{R}(-x)w(x)|dx|>0.\]
\item
Since $P(x)$ is positive almost everywhere on $i\RN$ we see that the polynomial $P(x)R(x)\ovl{R}(-x)$ is positive almost everywhere on $i\RN$. Since $w(x+1)$ is negative on $i\RN$ almost everywhere we have
\begin{multline*}T(-P(x-1)R(x-1)\ovl{R}(1-x))=-\int_{i\RN}R(x-1)\ovl{R}(1-x)P(x-1)w(x)|dx|=\\
-\int_{i\RN}R(x)\ovl{R}(-x)P(x)w(x+1)|dx|>0.
\end{multline*}
We used that $P(x)w(x+1)$ is holomorphic on $i\RN\times [0,1]$.
\end{enumerate}
The trace $T$ gives a positive definite form if and only if $T(R(x)\ovl{R}(-x))>0$ and $T(\Re_0 \biggl(R(x-1)\ovl{R}(1-x)P(x-1)\biggr))<0$. We have \begin{multline*} T(\Re_0 \biggl(R(x-1)\ovl{R}(1-x)P(x-1)\biggr))=\\
\frac12T(R(x-1)\ovl{R}(1-x)P(x-1)+R(1+x)\ovl{R}(-1-x)P(x+1))=\\
T(R(x-1)\ovl{R}(1-x)P(x-1)).
\end{multline*}
In the last equality we used that $T$ is invariant. The lemma follows.
\end{proof}

\subsection{Another proof of positivity}
\label{SubSubSecHarmonic}
I learned the following approach for constructing positive definite forms from Fedor Petrov. He told me a proof that a certain trace $T$ is positive on polynomials $F(x)=\ovl{R}(-x)R(x)$ and I extended this proof to the case $F(x)=-\Re_0(\ovl{R}(1-x)R(x-1)P(x-1))$. It follows that $T$ gives a positive definite form.

Recall that traces $T$ are in one-to-one correspondence with $T'\colon\CN[x]\to \CN$ such that $T'(PS)=0$ for all $S\in \CN[x]$ and $T'(1)=0$. If $F(x)=S(x+1)-S(x-1)$ then $T'(S)=T(F)$.

Suppose that $n$ is even and $P(x)=-(x-\alpha)(x-\beta)P_1(x)$, where $\alpha+\ovl{\beta}=0$, $-1<\Re \alpha\leq \Re \beta<1$, $P_1(x)$ is nonnegative when $\Re x=0$.  Consider $T'(R)=R(\beta)-R(\alpha)$ or $2R'(\alpha)$ in the case $\alpha=\beta$.  Without loss of generality we can assume that $\alpha,\beta\in\RN$, so $\alpha=-\beta$.

The condition of positivity is equivalent to $T$ being positive on $\ovl{R}(-x)R(x)$ and $-\Re_0(\ovl{R}(1-x)R(x-1)P(x-1))$ for all nonzero polynomials $R$.

We start with the first case: $F(x)=\ovl{R}(-x)R(x)$. This is equivalent to $F(x)\geq 0$ for $x\in i\RN$. Consider $S(x)$ such that $F(x)=S(x+1)-S(x-1)$. We can choose $S$ such that $S(-x)=-\ovl{S}(x)$. Consider $G(x)=S(-x)-S(\ovl{x})=-\ovl{S}(x)-S(\ovl{x})=-2\Re \ovl{S}(x)$. It follows that $G(x)$ is a harmonic function. For $x\in i\RN$ we have $-x=\ovl{x}$, so $G(x)=0$. For $x\in i\RN$ we have $\ovl{-1+x}=-1-x$, so $G(-1+x)=S(1-x)-S(-1-x)=F(-x)\geq 0$.

We see that $F(x)$ starts with $(-x^2)^k$. Hence $S(x)$ starts with $\frac{(-1)^k}{2(2k+1)}x^{2k+1}$. Therefore \[G(a+it)=S(-a-it)-S(a-it)=\frac{1}{2}(-2a)(t^{2k})+M(a,t),\] where $M(a,t)$ is a polynomial in two variables that has degree less than $2k$ in $t$. Since $G(it)=0$ we have $M(0,t)=0$, so $M(a,t)=aM_1(a,t)$. Hence \[G(a+it)=-a(t^{2k}-M_1(a,t)).\] It follows that for big $t$ and $a\in [-1,0)$ we have $G(a+it)>0$. 

We see that $G(x)\neq 0$ is nonnegative on the boundary of rectangle $[-1,0]\times [-iN,iN]$ for big enough $N$. Using the maximum principle we deduce that $G(x)$ is positive on $(-1,0)\times (-iN,iN)$. In particular, when $\alpha\neq 0$, $0<G(\alpha)=S(-\alpha)-S(\alpha)=S(\beta)-S(\alpha)=T(F)$.

Suppose that $\alpha=0$. Let $x=a+it$. We have $-\frac{\partial G}{\partial a}(x)=\frac{\partial S}{\partial a}(-x)+\frac{\partial S}{\partial a}(\ovl{x})=S'(-x)+S'(\ovl{x})$. For $x=0$ we get $-\frac{\partial G}{\partial a}(0)=2S'(0)=T(F)$. Using the maximum principle and Hopf lemma we deduce that $-\frac{\partial G}{\partial a}(0)>0$, so $T(F)>0$.

Now we deal with the case $F(x)=-\Re_0(\ovl{R}(1-x)R(x-1)P(x-1))$. Consider \[G_0(x)=S(-\ovl{x})-S(x)=-2\Re S(x).\] \[G_1(x)=\Re(\ovl{R}(-x)R(x)P(x)).\] \[G(x)=G_0(x)+G_1(x).\] The functions $G_0$, $G_1$ and $G$ are harmonic.

Recall that for $S(x)\in \CN[x]$, $\Re_0 S(x)$ is an element of $\CN[x]$ such that $\Re(S(x)|_{i\RN})=(\Re_0 S(x))|_{i\RN}$. 

For $x\in i\RN-1$ we have $G_0(x)=S(-\ovl{x})-S(x)=S(x+2)-S(x)=F(x+1)$. For $x\in i\RN$ we have \[G_1(x-1)=\Re(\ovl{R}(1-x)R(x-1)P(x-1))=\Re_0(\ovl{R}(1-x)R(x-1)P(x-1))=-F(x).\] Therefore $G_1(x)=-F(x+1)$ for $x\in i\RN-1$. Hence $G(x)=0$ for $x\in i\RN-1$.

For $x\in i\RN$ we have \[G_0(x)=S(-\ovl{x})-S(x)=0,\] \[G_1(x)=\Re(\ovl{R}(-x)R_2(x)P(x))=|R(x)|^2P(x)\geq 0\] since $P(x)\geq 0$ for $x\in i\RN$.

Since $\ovl{R}(-x)R(x)P(x)\geq 0$ when $x\in i\RN$ it has leading term $c(-x^2)^k$ for some $c>0, k\in\ZN_{>0}$. Hence \[F(x)=-\Re_0(\ovl{R}(1-x)R(x-1)P(x-1))\] has leading term $-c(-x^2)^k$. It follows that $S(x)$ has leading term $(-1)^{k+1} c\frac{x^{2k+1}}{2(2k+1)}$. 

For $x=a+it$ we have \[G_0(x)=S(-a+it)-S(a+it)=(-2a)(\frac{(-1)^{k+1}}2c (it)^{2k}+M_0(a,t)),\] where $M_0$ has degree less than $2k$ in $t$. We also have \[G_1(a+it)=(-1)^{k}\Re c(a+it)^{2k}+\cdots=ct^{2k}+M_1(a,t),\] where $M_1$ has degree less than $2k+2$ in $t$. Hence \[G(a+it)=G_0(a+it)+G_1(a+it)=(a+1)ct^{2k}+M(a,t),\] where $M(a,t)$ has degree less than $2k$ in $t$. Since $G(it-1)=0$ we have $M(a,t)=(a+1)N(a,t)$, so \[G(a+it)=(a+1)(ct^{2k+2}+N(a,t)).\] 

It follows that $G(a+it)>0$ for $a\in (-1,0]$ and big enough $t$. Since $G(x)\geq 0$ for $x\in i\RN$ or $x\in i\RN+1$ we deduce that for big enough $N>0$ $G(x)$ is nonnegative on the boundary of rectangle $[-1,0]\times [-N,N]$.

Suppose that $\alpha<0$. Using the maximum principle we deduce that $G(\alpha)>0$. We have $G(\alpha)=G_0(\alpha)+G_1(\alpha)=S(\beta)-S(\alpha)+\Re \ovl{R}(-\alpha)R(\alpha)P(\alpha)=T(F)+0=T(F)$.

Suppose that $\alpha=0$. Similarly to the above we use the maximum principle and Hopf lemma to deduce that $T(F)>0$.

So we proved that $T(F)>0$ when $F=R(x)\ovl{R}(-x)$ or $F(x)=-\Re_0(R(x-1)R(1-x)P(x-1)$. Using Proposition~\ref{StatHowCLooks} we deduce that $T$ gives a positive definite form on $\mc{A}$.

\begin{rem}
\label{RemManyForms}
Suppose that $P(x)=(-x^2)^m+\cdots$  has $2m$ distinct roots such that for any root $\alpha$ we have $0<|\Re\alpha|<1$. It follows that there are $m$ distinct pairs of roots $(\alpha,\beta)$ such that $\alpha+\ovl{\beta}=0$. It follows from an analytic formula in subsection~\ref{SubSecAnalytic} or the proof in subsection~\ref{SubSubSecHarmonic} that each pair gives a positive definite form. It is not hard to see that these positive definite forms are linearly independent. It follows that in this case the cone of positive definite forms has dimension at least $m$.
\end{rem}

\section{Appendix}
\subsection{Dual bimodule.}
\label{SubSubSecDualBimod}
Suppose that $M$ is a Harish-Chandra $\mc{A}$-bimodule, $M^*$ is the dual space. Then $M^*$ has a natural structure of an $\mc{A}$-bimodule. Let $M^{\vee}$ be the set of elements $v\in M^*$ such that $\ad e$ and $\ad f$ act locally nilpotently on $v$, $\ad h$ acts locally finitely. It is easy to see that $M^{\vee}$ is a subbimodule in $M^*$.



An invariant Hermitian form on $M$ gives an $\mc{A}$-bimodule homomorphism from $\ovl{M}$ to $M^{\vee}$, where $a\in \mc{A}$ acts on $\ovl{M}$ as $\ovl{a}$. So $M^{\vee}$ can be a useful object for the classification of invariant positive definite forms.

We will use $M^{\vee}$ in the case when $n=2$. In this case $\ad e$, $\ad f$, $\ad h$ form a Lie subalgebra $\mf{sl}_2$.

Suppose that $M=\oplus_{i\in I}U_i$ is a decomposition of $M$ as an $\mf{sl}_2$-module. We have $M^*=\prod_{i\in I} U_i$. $I$ is either finite or countable and in the second case the dimension of $U_i$ tends to infinity. Suppose that $x$ is an element of $U_i$, $U_i$ has highest weight $k$, $2l+2<k$. It is easy to see that $(\ad e)^l x$ and $(\ad f)^l x$ cannot both be zero.  It follows that $M^{\vee}=\oplus_{i\in I} U_i$. We deduce that $M^{\vee}$ is a \HC bimodule.

Recall how irreducible Harish-Chandra bimodules decompose as an $\mf{sl}_2$-module. Denote by $V_i$ the irreducible module of highest weight $i$.

\begin{enumerate}
\item
$\mc{A}=\oplus_{i=0}^{\infty} V_{2i}$
\item
In the case $\lambda\in \ZN\setminus\{-1\}$ there are two irreducible bimodules, $\End(W)$ and $I$. We have $\End(W)=\oplus_{i=0}^{\lambda} V_{2i}$, $I=\oplus_{i=\lambda+1}^{\infty} V_{2i}$.
\item
In the case $\lambda\in \ZN+\frac{1}{2}$ we can assume that $\lambda\geq -\frac{1}{2}$. We have two irreducible bimodules $\mc{A}$ and $M$. Then $M=\oplus_{i=1}^{\infty} U_{2i+2\lambda}$.
\end{enumerate}
If $M$ is an irreducible Harish-Chandra bimodule then it is easy to deduce from the $\mf{sl}_2$-decomposition that $\ovl{M}\cong M\cong M^{\vee}$. Therefore
\begin{prop}
\label{StatUniqueHermitianForm2}
Suppose that $n=2$, $M$ is an irreducible Harish-Chandra bimodule with $\lambda\in\RN$. Then there is a unique invariant Hermitian form on $M$.
\end{prop}
\begin{rem*} If $\lambda\in \ZN\setminus\{-1\}$ then $\mc{A}$ has a finite-dimensional representation $V$ and we get a short exact sequence \[0\to I=\operatorname{Ann}(V)\to \mc{A}\to \End(V)\to 0.\] Applying ${\cdot}^{\vee}$ we get a short exact sequence \[0\to\End(V)\to \mc{A}^{\vee}\to I\to 0.\]
Since the inclusion $I\subset \mc{A}$ does not split maps from $\mc{A}$ to $\mc{A}^{\vee}$ factor through $\End(V)$. It follows that every invariant hermitian form on $\mc{A}$ is zero on $I$. So in this case we get a classification of invariant hermitian forms on $\mc{A}$ without doing any computations.
\end{rem*}

\subsection{Index}
Suppose that $f$ is a polynomial such that $0\notin f(i\RN)$. Then we define index of $f$ with respect to zero as $\Ind f=\frac{-1}{\pi}\int\limits_{-i\infty}^{i\infty}\Im(\frac{f'}{f})dx$. We have $\Ind {fg}=\Ind f+\Ind g$, $\Ind (x-a)=-\sign\Re a$. Index has a geometric interpretation. Consider a continuous choice of $\arg f(x)|_{i\RN}$. There exist $\lim_{x\to\pm i\infty}\arg f(x)$ and $\pi\Ind {f(x)}=\lim_{x\to i\infty}\arg f(x)-\lim_{x\to -i\infty}\arg f(x)$.

Recall that $\rho_{< a}(F(x))$ is the number of roots with multiplicities of $F(x)$ with real part less than $a$, similarly for other inequality signs.
\begin{lem}
\label{LemWindingNumber}
Suppose that $\Re F(x)\geq 0$ when $\Re x=a$. Denote by $k$ the number of roots of $F(x)$ that have real part equal to $a$ and odd multiplicity. Then 
\begin{enumerate}
\item
\[\rho_{>a}(F(x))\leq \rho_{< a}(F(x))+k+1\]
\[\rho_{<a}(F(x))\leq \rho_{> a}(F(x))+k+1\]
In particular,
\[\rho_{>a}(F(x))\leq \rho_{\leq a}(F(x))+1\]
\[\rho_{<a}(F(x))\leq \rho_{\geq a}(F(x))+1\]
\item
Suppose that $\rho_{>a}(F(x))= \rho_{< a}(F(x))+k+1$. Denote the degree of $F$ by $2d-1$. Then the leading coefficient of $F(x)$  has sign $(-1)^d$.

Suppose that $\rho_{<a}(F(x))=\rho_{> a}(F(x))+k+1$. Then the leading coefficient of $F(x)$  has sign $(-1)^{d-1}$.
\end{enumerate}
\end{lem}
\begin{proof}
\begin{enumerate}

\item
Suppose that $\rho_{<a}(F(x))=l$, $\rho_{>a}(F(x))=m$, $\rho_{=a}(F(x))=s$. Denote by $Q(x)$ the polynomial with the top degree coefficient $i^s$ such that the roots of $Q(x)$ are all the roots of $F(x)$ with real part $a$ with the same multiplicities. Denote $\frac{F(x)}{Q(x)}$ by $G(x)$. 

We see that $Q(x)$ is real on $\Re x=a$ and changes sign at most $k$ times. Hence $\Re G(x)$ changes sign at most $k$ times on $\Re x=a$. Using the geometric interpretation of index we see that \[-k-1\leq \Ind {G(x-a)}\leq k+1.\] On the other hand $\Ind {G(x-a)}=m-l$. Hence $|m-l|\leq k+1$.
\item
Suppose that $\Ind {G(x-a)}=k+1$. This corresponds to $l=k+m+1$. Using the geometric interpretation of index we see that \[\lim_{x\to a-i\infty} \frac{G(x)}{|G(x)|}=-i,\]\[\lim_{x\to a+i\infty} \frac{G(x)}{|G(x)|}=i^{2s+1}.\] On the other hand we have \[\lim_{x\to a-i\infty}\frac{Q(x)}{|Q(x)|}=1.\] Hence \[\lim_{x\to a-i\infty}\frac{F(x)}{|F(x)|}=-i.\] If $F(x)$ starts with $ax^{2d-1}$ then \[\lim_{x\to a-i\infty}\frac{F(x)}{|F(x)|}=\frac{a}{|a|}(-i)^{2d-1}=(-i)\frac{a}{|a|}(-1)^{d-1}.\] Hence $\frac{a}{|a|}(-1)^{d-1}=1$, so $a$ is real and has sign $(-1)^{d-1}$. The other case is done similarly.
\end{enumerate}
\end{proof}


\subsection{Good approximations}

 Denote by $M$ the monoid of nonzero polynomials that are nonnegative on the line $\Re x=\frac12$ with respect to multiplication. A polynomial $F(x)\in M$ has even degree $2d$ and real leading coefficient of sign $(-1)^d$.

Let $a<\frac{\pi}{2}$. We say that a polynomial $F(x)\in M$ has $a$-{\it bounded} argument if $\arg F|_{i\RN}\subset (-a,a)$.  We say that a polynomial $F(x)\in M$  with $a$-bounded argument has $\eps$-small argument if $\arg F|_{i(-\frac{1}{\eps},\frac{1}{\eps})}\subset (-\eps,\eps)$. For $P(x)\in M$ we say that $\frac{1}{P(x)}$ has a good approximation if there exists $a<\frac{\pi}{2}$ such that for any $\eps>0$ there exists $F(x)$ such that $P(x)F(x)$ has $a$-bounded $\eps$-small argument. 

The motivation for this definition is as follows:
\begin{prop}
\label{PropFromGoodApproxToNonUnit}
Suppose that a deformation $\mc{A}$ of $\CN[x,y]^{C_n}$ has parameter $P(x)$ such that $P(x)$ is positive on $\Re x=0$ and $\frac{1}{P(2x-1)}$ has a good approximation. Then $\mc{A}$ is not unitarizable.
\end{prop}
\begin{proof}
Take $F(x)\in M$ such that $P(2x-1)F(x)$ has $a$-bounded argument for some $0<a<\frac{\pi}{2}$. It follows that $\Re P(2x-1)F(x)\geq 0$ when $\Re x=0$. By definition $\Re F(x)=F(x)\geq 0$ when $\Re x=\frac{1}{2}$. It follows that $\Re F(\frac{x+1}{2})\geq 0$ and $\Re P(x-1)F(\frac{x}{2})\geq 0$ when $\Re x=0$. Hence by Proposition~\ref{StatWhenCcontainsZero} $C$ contains zero in this case. In particular $\mc{A}$ is not unitarizable.
\end{proof} 

\begin{prop}
\label{StatGoodApproximationOfProduct}
Suppose that for $P_1,\ldots,P_k\in M$ $\frac{1}{P_1(x)}$, $\frac{1}{P_2(x)}$,\ldots, $\frac{1}{P_k(x)}$ have a good approximation. Then $\frac{1}{P_1(x)\ldots P_k(x)}$ has a good approximation.
\end{prop}
\begin{proof}

It is enough to prove the statement for $k=2$.

Let $a_i$ be the number we get from definition of good approximation of $\frac{1}{P_i(x)}$ for $i=1,2$, $a=\max(a_1,a_2)$. It follows that for any $\eps_1,\eps_2>0$ there exist $F_1,F_2\in M$ such that $F_i(x)P_i(x)$ has $a$-bounded $\eps_i$-small argument.

Let $\eps_0>0$ be number such that $a+\eps_0<\frac{\pi}{2}$. Let $\eps_0>\eps>0$. It is enough to prove that there exists $F\in M$ such that $F(x)P_1(x)P_2(x)$ has $a+\eps$-bounded $\eps$-small argument.

First choose $F_1(x)$ such that $F_1(x)P_1(x)$ has $a$-bounded $\frac{\eps}{2}$-small argument. Since $F_1(x)$ tends to $\infty$ when $x$ tends to $\pm i\infty$ there exists $\eps_1>0$, $\eps_1<\frac{\eps}{2}$ such that $\arg F_1(x)\in (-\eps,\eps)$ for $x\in i\RN$, $|x|\geq \frac{1}{\eps_1}$. Choose $F_2(x)$ such that $F_2(x)P_2(x)$ has $a$-bounded $\eps_1$-small argument. 

Let us prove that $F_1(x)F_2(x)P_1(x)P_2(x)$ has $\eps$-small $a+\eps$-bounded argument. Denote $F_1(x)F_2(x)$ by $F(x)$. There are three cases for $x\in i\RN$:
\begin{enumerate}
\item
$|x|<\frac{1}{\eps}$. In this case $\arg F_1(x)P_1(x)\in (-\frac{\eps}{2},\frac \eps 2)$, $\arg F_2(x)P_2(x)\in (-\eps_1,\eps_1)$, hence $\arg F(x)\in (-\eps,\eps)$. This proves that $F(x)P(x)$ has $\eps$-small argument.
\item
$|x|<\frac{1}{\eps_1}$. In this case $\arg F_1(x)P_1(x)\in (-a,a)$, $\arg F_2(x)P_2(x)\in (-\eps_1,\eps_1)$, hence $\arg F(x)P(x)\in (-a-\eps_1,a+\eps_1)$, so it is $a+\eps$-bounded.
\item
$|x|\geq \frac{1}{\eps_1}$. In this case $\arg F_1(x)P_1(x)\in (-\eps,\eps)$, $\arg F_2(x)P_2(x)\in (-a,a)$, hence $\arg F(x)P(x)\in (-a-\eps,a+\eps)$, so it is $a+\eps$-bounded.
\end{enumerate}

Hence $F(x)P(x)$ has $\eps$-small $a+\eps$-bounded argument.

\end{proof}
Now let us prove the following proposition.

\begin{prop}
Polynomial $P(x)=-(x-a)(x-1+\ovl{a})$ has a good approximation when $\Re a<0$.
\end{prop}
\begin{proof}
Making linear change of coordinates $x\mapsto x+ir$ we can assume that $a\in \RN$. Denote $1-a$ by $b$. We have $P(x)=-(x-a)(x-b)$. We see that $P(x)$ is positive on the line $\Re x=\frac12$.

Note that we can define the notion of $\eps$-bounded and $\eps$-small argument for an entire function that is positive on $\Re x=\frac12$. At first we will find an entire function $F_0$ such that $F_0(x)>0$ when $\Re x=\frac{1}{2}$ and $F_0(x)P(x)$ has $\eps$-bounded, hence $\eps$-small argument. After that we will  approximate $F_0$ with polynomials.

Fix $\eps>0$. Consider $R_0(x)=(e^{n(x-a)}-1)(e^{n(b-x)}-1)$, where $n$ is a positive integer that we will specify later. We see that $R_0(a)=R_0(b)=0$, so $F_0(x)=\frac{R_0(x)}{P(x)}$ is a holomorphic function. For $\Re x=\frac{1}{2}$ $R_0(x)$ is a product of two conjugate nonzero complex numbers, so $R_0(x)>0$. Since $P_0(x)$ is positive on $\Re x=\frac{1}{2}$, $F_0(x)$ is positive on $\Re x=\frac{1}{2}$.

Suppose that $x\in i\RN$. We have $R_0(x)=e^{n(b-a)}-e^{n(x-a)}-e^{n(b-x)}+1$. We see that \[|\Im R_0(x)|\leq |e^{n(x-a)}+e^{n(b-x)}|\leq e^{-na}+e^{nb},\] \[|R_0(x)|\geq e^{n(b-a)}-e^{-na}-e^{-nb}-1.\] We see that for big enough $n$ we have \[\frac{|\Im R_0(x)|}{|R_0(x)|}\leq \frac{\eps}{2},\] hence the argument of $R_0(x)$ belongs to $(-\eps,\eps)$.

The function $e^{x-a}$ is approximated uniformly on compact sets by $E^-_l(x)=(\frac{l+x-a}{l})^l$. Similarly $e^{b-x}$ is approximated by $E^+_l(x)=(\frac{l+b-x}{l})^l$. 

Let us prove that for some sequence $l_1\leq l_2\leq\ldots\leq l_n$ we can take \[R(x)=(R^+(x)-1)(R^-(x)-1)=(E^+_{l_1}(x)\cdots E^+_{l_n}(x)-1)(E^-_{l_1}\cdots E^-_{l_n}(x)-1).\] We see that $R$ is positive on $\Re x=\frac12$ and $R(a)=R(b)=0$, so that $S(x)=\frac{R(x)}{P(x)}$ is a polynomial that is positive on the line $\Re x=\frac{1}{2}$.

Sequence $(1-\frac{a}{l})^l$ tends to $e^{-a}>1$. It follows that there exists $c>1$ and $l_0$ such that for $l\geq l_0$ $(1-\frac{a}{l})^l>c$. We note that when $x\in i\RN$, $|\frac{l+x-a}{l}|\geq 1-\frac{a}{l}$. Therefore when $x\in i\RN$, $l\geq l_0$ $|E^-_l(x)|>c$. We deduce that when $l_0\leq l_1\leq\cdots\leq l_n$ we have $|R^-(x)|>c^n$. Choosing another $c>1$ if necessary we similarly prove that $|R^+(x)|>c^n$ when $x\in i\RN$. 

It follows that the argument of $R(x)=(R^+(x)-1)(R^-(x)-1)$ differs from the argument of $R^+(x)R^-(x)$ by at most $\frac{4\pi}{c^n}$. Fix $n$ such that $1+\frac{4\pi}{c^n}<\frac{\pi}{2}$. If we prove that the argument of $R^+(x)R^-(x)$ is $1$-bounded it will follow that the argument of $R(x)$ is $1+\frac{4\pi}{c^n}$-bounded.

Suppose that $K$ is a compact subset of $\CN$. When $l_1,\ldots,l_n$ tend to infinity $\max_{x\in K}|R(x)-R_0(x)|$ tends to zero. Function $R_0(x)$ has $\eps$-bounded argument. Taking $K=[-\frac{1}{2\eps}i,\frac{1}{2\eps}i]$ we deduce that for big enough $l_1,\ldots,l_n$ $R(x)$ has $2\eps$-small argument.

Arguing as in proof of Proposition~\ref{StatGoodApproximationOfProduct} it is enough to prove that for any $\eps_1>0$ we can choose sufficiently large $l$ such that $E^+_lE^-_l$ has $\eps_1$-small $\frac{1}{2}$-bounded argument. For big enough $l$ the product $E^+_lE^-_l$ approximates $e^{b-a}$ on $[-\frac{i}{\eps_1},\frac{i}{\eps_1}]$, hence it has $\eps_1$-small argument. We have $E^+_lE^-_l=(\frac{(l+x-a)(l+b-x)}{l^2})^l$. We see that \[\tan\arg(l+x-a)(l+b-x)=\frac{x(b+a)}{x^2+(l-a)(l+b)}.\] We have \[\frac{x(b+a)}{x^2+(l-a)(l+b)}\leq \frac{x\cdot 1}{2x\sqrt{(l-a)(l+b)}}\leq \frac{1}{2l}.\] We used that $a+b=1$, $a<0$, $b>0$. Hence the argument of $(l+x-a)(l+b-x)$ belongs to $(-\frac{1}{2l},\frac{1}{2l})$, so the argument of $E^+_lE^-_l$ belongs to $(-\frac{1}{2},\frac{1}{2})$. The statement follows.
\end{proof}
Using Proposition~\ref{StatGoodApproximationOfProduct} we deduce the following
\begin{cor}
\label{CorGoodApprox}
Suppose that $P(x)\in \CN[x]$ is positive on $\Re x=\frac{1}{2}$ and does not have roots in the set $0\leq \Re x\leq 1$. Then $\frac1{P(x)}$ has a good approximation.
\end{cor}

\textsc{Department of Mathematics, MIT, 77 Mass. Ave, Cambridge, MA 02139}

{\it E-mail address}: \texttt{\href{mailto:klyuev@mit.edu}{klyuev@mit.edu}}
\end{document}